\documentclass[11pt]{article}
\usepackage{etex}
\usepackage[utf8]{inputenc}
\usepackage[all]{xy}
\usepackage{amsfonts}
\usepackage{amsthm}
\usepackage{enumerate}
\usepackage{graphicx}
\usepackage{mathrsfs}
\usepackage{bm}
\usepackage[noadjust]{cite}
\usepackage{amssymb,amsmath} 
\usepackage{makecell}
\usepackage{tikz}
\usepackage{pgf}
\usepackage{tikz}
\usetikzlibrary{patterns}
\usepackage{pgffor}
\usepackage{pgfcalendar}
\usepackage{pgfpages}
\usepackage{shuffle,yfonts}
\usepackage{mathtools}

\DeclareFontFamily{U}{shuffle}{}
\DeclareFontShape{U}{shuffle}{m}{n}{ <-8>shuffle7 <8->shuffle10}{}

\mathtoolsset{showonlyrefs}

\newcommand\ga{{\alpha}}

\newcommand\eps{{\varepsilon}}
\newcommand{\calB}{{\mathcal B}}
\newcommand{\bfk}{{\boldsymbol{\sl{k}}}}
\newcommand{\bfm}{{\boldsymbol{\sl{m}}}}
\newcommand{\bfn}{{\boldsymbol{\sl{n}}}}

 \allowdisplaybreaks
\usetikzlibrary{arrows,shapes,chains}

\allowdisplaybreaks

\DeclareMathOperator*{\dep}{dep}
\DeclareMathOperator{\Li}{Li}
\DeclareMathOperator{\Res}{Res}
\DeclareMathOperator{\Gr}{Gr}
\DeclareMathOperator{\Nor}{\text{\rm N\"o}}

\def\N{\mathbb{N}}

\def\Q{\mathbb{Q}}
\def\CC{\mathbb{C}}

\theoremstyle{plain}
\newtheorem{thm}{Theorem}[section]
\newtheorem{lem}[thm]{Lemma}
\newtheorem{cor}[thm]{Corollary}
\newtheorem{con}[thm]{Conjecture}

\theoremstyle{definition}

\newtheorem{re}[thm]{Remark}

\begin{document}
\title{\bf Rational Approximations for Reciprocals of Multiple Zeta Values and Trivariate Cauchy Numbers}
\author{
{Ce Xu${}^{a,}$\thanks{Email: cexu2020@ahnu.edu.cn, ORCID 0000-0002-0059-7420.}{} \ \ and Jianqiang Zhao${}^{b,}$\thanks{Email: zhaoj@ihes.fr, ORCID 0000-0003-1407-4230, corresponding author.}}\\[1mm]
\small a. School of Mathematics and Statistics, Anhui Normal University, Wuhu 241002, PRC\\
\small b. Department of Mathematics, The Bishop's School, La Jolla, CA 92037, USA}

\date{}
\maketitle

\noindent{\bf Abstract.}
In this paper, we will study a trivariate extension of the Cauchy numbers of both the first kind (also called Gregory coefficients) and the second kind (also called N\"orlund numbers) via the Laurent expansion of the reciprocal of any positive integer power (which is called the order) of multiple polylogarithms. In the case of logarithm, we will show by the WZ method that for each order $\ell>1$ some Gregory coefficient of order $\ell$ must vanish, in contrast to the fact that all classical Gregory coefficients are nonzero. We also prove in this higher order logarithm case that the sequence is eventually alternating for each fixed order, a property enjoyed by the classical Gregory coefficients. In the most general setting, we conjecture that these new sequences are all eventually positive, which is supported by strong numerical evidence. Finally, we confirm this conjecture in the special case of polylogarithms and double polylogarithms. As a by product, for each zeta value and double zeta value, we find an infinite family of identities expressing its reciprocal as a sum of a rational number and an improper integral.

\medskip

\noindent{\bf Keywords}: Rational approximation; Cauchy numbers; Gregory coefficients; N\"orlund numbers; multiple polylogarithms; multiple zeta values; contour integration.
\medskip

\noindent{\bf AMS Subject Classifications (2020):} 11B83, 30E20, 11M32, 11G55, 65H04, 05A10.



\section{Introduction}

The study of various constants and sequences has been a major theme in mathematical research for centuries. Many prominent mathematicians have their names attached to some of them, such as Euler, Gauss, Ramanujan, and others. We invite readers to explore Sloan's online encyclopedia of integer sequences \cite{Sloane2025} for many interesting examples. In this paper, we first identify a common extension of two such sequences, namely the Cauchy numbers of the first and second kinds, and further generalize them to higher orders. We then incorporate multiple polylogarithm functions into this framework for the first time and define a trivariate version as a vast extension of these families. This trivariate generalization of Cauchy numbers, together with Panzer's parity results \cite{Panzer2017}, yields rational approximations to the reciprocals of multiple zeta values, with explicit formulas for the single and double cases. We hope these approximations will shed new light on the irrationality of multiple zeta values, a problem that has remained largely untouched except for the odd Riemann zeta values. Recent developments on the irrationality of odd Riemann zeta values appear in Lai-Yu \cite{LY2020} and Fischler \cite{Fischler2026} and references therein. Moreover, recent results on the irrationality of $p$-adic zeta values can be found in Lai-Sprang-Zudilin \cite{LSZ2026} and the references cited there.

\subsection{Cauchy numbers of the first kind}
Define the \emph{Gregory coefficients} (aka \emph{Cauchy numbers of the first kind}) $\{G_n\}_{n\geq 1}$ by the generating function
\begin{equation}\label{defn-Gregorycoefficients1}
\frac{x}{\log(1+x)}=1+\sum_{n=1}^\infty G_nx^n.
\end{equation}

Gregory was among the first ones, if not the first, to introduce the numbers $G_n$ in developing his method for numerical integration (see \cite{Phillips1972}). In addition to Gregory, these numbers have been investigated by a distinguished line of mathematicians that includes Mascheroni, Cauchy, Schr\"oder, and many others. Moreover, they have been given some different names such as the Bernoulli numbers of the second kind \cite{Bl2017,Roman1984} and the Cauchy numbers of the first kind \cite{CC2012,Komatsu2013,MSC2006}. Many of their properties have been discovered and then rediscovered repeatedly. In particular, Komatsu generalized these to the so-called poly-Cauchy numbers in \cite{Komatsu2013}. Incidentally, other generalizations have been proposed; for instance, Matsusaka, Murahara, and Onozuka \cite{MMO2025} defined multidimensional generalized Gregory coefficients $G_{m,n}$, and Ishii and Shinohara \cite{ISSH2025} studied their relations with multiple zeta functions at non-positive integers.

It is well-known that all $G_n$'s are rational numbers starting with
\begin{align*}
G_1=\frac1{2},\ G_2=-\frac1{12},\ G_3=\frac1{24},\ G_4=-\frac{19}{720}, \ldots.
\end{align*}
One can use \cite{Sloane2025} (see OEIS A002206 and A002207) to find more terms of these numbers.
They have been shown to be intimately related Euler's constant by Mascheroni \cite[pp. 21-23]{Mascheroni}, a fact that has been extended to some ``finite'' version in \cite{KanekoMatsusakaSeki2025}. Among other important properties, the Gregory coefficients are equipped with alternating signs, or more precisely, $(-1)^{n-1}G_n>0$ for $n\geq 1$ (see Remark~\ref{rem:nonZero}), which we will generalize in this paper.

For all $\ell\ge 1$, we extend \eqref{defn-Gregorycoefficients1} to define the \emph{Gregory coefficients of order $\ell$}, denoted by $G_n^{(\ell)}$,  as follows (see \cite{XuZhao2026}):
\begin{align}\label{defn-Gregorycoefficientsp}
\left(\frac{x}{\log(1+x)}\right)^\ell =1+\sum_{n=1}^\infty G_n^{(\ell)}x^n.
\end{align}
These numbers behave very much like their classical counterpart at order 1. However,
we will see in Theorem~\ref{thm:vanishing} that, unlike the classical case where all $G_n$'s are non-zero (see \eqref{equ:nonZero}), some particular higher order Gregory coefficient always vanishes for each order $\ell$. The proof relies on the WZ-method in the odd order cases.

The second main result of this paper is the following theorem which says that, for every order $\ell\ge 1$, the alternating pattern of higher order Gregory coefficients $G_n^{(\ell)}$ holds for all sufficiently large $n$, generalizing the corresponding result for the classical Gregory coefficients (see Remark~\ref{rem:nonZero}). For convenience, we say a power series $\sum{}_{n\ge 0} a_n x^n$ has eventually alternating coefficients if $(-1)^n a_n>0$ for all $n\gg 0$ or $(-1)^n a_n<0$ for all $n\gg 0$.

\begin{thm}\label{thm:GregoryGeneralOrderA} \emph{(=Theorem \ref{thm:GregoryGeneralOrder})}
For each order $\ell\ge 1$, there is some positive integer $\Gr(\ell)$ such that for all $n\ge \Gr(\ell)$
\begin{equation*}
(-1)^{n-1} G_n^{(\ell)}>0.
\end{equation*}
Therefore, the Maclaurin series of $x^\ell/\log^\ell(1+x)$ has eventually alternating coefficients.
\end{thm}

Motivated by Schr\"oder's approach in \cite{Schr1880}, our main idea is to transform the computation of $G_n^{(\ell)}$ to some integrals by contour integration using complex analysis. By cutting these integrals into two parts appropriately with opposite signs, we show that the positive part always dominates the negative one, thereby proving that $(-1)^{n-1}G_n^{(\ell)}>0$, for all sufficiently large $n$.

\subsection{Cauchy numbers of the second kind}
We define the \emph{(normalized) N\"orlund numbers} (aka \emph{Cauchy numbers of the second kind}) by the generating function
\begin{equation}\label{defn-Norlund1}
\frac{x}{(1+x)\log(1+x)}=1+\sum_{n=1}^\infty N_n x^n .
\end{equation}
We remark that traditionally the N\"orlund numbers are defined by $n!N_n$ using an exponential generating function.
These numbers are called N\"orlund numbers in \cite{Howard1993,Young2008} since N\"orlund apparently studied these in \cite[pp.\ 150-151]{Norlund1924} while some others call them the Cauchy numbers of the second kind \cite{C1974,Komatsu2015}.

We may mimic the definition for higher order Gregory coefficients by defining \emph{higher order N\"orlund numbers ${\tilde N}_n^{(\ell)}$ of the first kind}
\begin{align}\label{defn-Norlund-first}
 \left(\frac{x}{(1+x)\log(1+x)}\right)^\ell =1+\sum_{n=1}^\infty {\tilde N}_n^{(\ell)}x^n.
\end{align}
Then the following result is an immediate consequence of the alternating properties of the N\"orlund numbers.

\begin{thm}\label{thm:Norlund1A} \emph{(=Theorem~\ref{thm:Norlund1})}
For each order $\ell\ge 1$ and all $n\ge 1$, we have
\begin{equation*}
(-1)^{n} {\tilde N}_n^{(\ell)}>0.
\end{equation*}
\end{thm}

We may also define the \emph{higher order N\"orlund numbers ${\tilde N}_n^{(\ell)}$ of the second kind} by
\begin{align}\label{defn-Norlund-second}
\frac{1}{1+x}\cdot \left(\frac{x}{\log(1+x)}\right)^\ell =1+\sum_{n=1}^\infty N_n^{(\ell)}x^n.
\end{align}
By essentially the same ideas used in Section~\ref{sec:GregoryGeneralOrder}, we can prove (see Theorem~\ref{thm:Norlund2}) that the
$N_n^{(\ell)}$ are eventually alternating for any fixed $\ell\ge 1$.

Furthermore, we can even increase the power of $(1+x)$ to any positive integer and define, for any $j,\ell\ge 1$, the
\emph{higher order N\"orlund-type numbers} $N_n^{j,\ell}$ by
\begin{align}\label{defn-Norlund-type}
\frac{1}{(1+x)^j} \left(\frac{x}{\log(1+x)}\right)^\ell =1+\sum_{n=1}^\infty N_n^{j,\ell} x^n.
\end{align}
We will prove in Theorem~\ref{thm:NorlundType} that the power series above is eventually alternating by induction on $j$ with the initial case given by Theorem~\ref{thm:Norlund2}.

\subsection{A trivariate extension}
In recent years, the study of multiple polylogarithms and their special values have attracted the attention of many mathematicians and theoretical physicists due to their important and sometime surprising relations to a number of areas in mathematics and physics (see \cite{ZhaoBook}, in particular the historical notes at the end of each chapter).

For each finite sequence of positive integers $\bfk=(k_1,\dots,k_d)$ (called a \emph{composition} of \emph{weight} $|\bfk|=k_1+\dots+k_d$), we recall that the \emph{(single variable) multiple polylogarithm} is defined by
\begin{equation}\label{equ:Polylog}
\Li_\bfk(z) =\sum_{n_1>\dots>n_d>0} \frac{z^{n_1}}{n_1^{k_1}\dots n_d^{k_d}}    \quad(|z|\le 1, (k_1,z)\ne (1,1)).
\end{equation}
Here, we call $\dep(\bfk):=d$ the depth.
If $\bfk$ is admissible, i.e., $\bfk=(k_1,\dots,k_d)\in \N^d$ with $k_1>1$, then
the \emph{multiple zeta value}
\begin{equation*}
    \zeta(\bfk):=\Li_\bfk(1)=\sum_{n_1>\dots>n_d>0} \frac{1}{n_1^{k_1}\dots n_d^{k_d}} .
\end{equation*}

Noticing that $\log(1-x)=-\Li_1(x)$ is essentially the simplest multiple polylogarithm, for any $j\ge 0$, $\ell\ge 1$ and
$\bfk=(k_1,\dots,k_d)\in \N^d$, we define the numbers $C_{n}^{j;\bfk;\ell}$ by
\begin{equation}\label{equ:trivariateDefn}
 \frac{1}{(1-x)^j}\cdot \left(\frac{x^{\dep(\bfk)}}{\Li_\bfk(x)}\right)^\ell =\sum_{n=0}^\infty C_{n}^{j;\bfk;\ell}x^n.
\end{equation}
where the constant term
\begin{equation*}
C_{0}^{j;\bfk;\ell}=\left(\prod_{m=1}^d (d+1-m)^{k_m}\right)^\ell.
\end{equation*}

By specializing at $\bfk=(1)$ we see that
\begin{equation*}
 \frac{1}{(1-x)^j}\cdot \left(\frac{-x}{\log(1-x)}\right)^\ell =1+\sum_{n=1}^\infty C_{n}^{j;1;\ell}x^n
\end{equation*}
provides a common generalization of the higher order Gregory coefficients
and the higher order N\"orlund-type numbers since
\begin{equation*}
  G_n^{(\ell)}=(-1)^{n} C_{n}^{0;1;\ell}\quad\text{and}\quad
  N_n^{j,\ell}=(-1)^{n} C_{n}^{j;1;\ell}.
\end{equation*}

In fact, from over-whelming evidence, we are confident that the following conjecture should be true.
For convenience, we say a power series $\sum_{n\ge 0} a_n x^n$ has eventually positive (resp. negative)
coefficients if $a_n>0$ (resp. $a_n<0$)  for all $n\gg 0$.

\begin{con}\label{conj:Trivar}
For any fixed integers $j,\ell\ge 1$ and composition $\bfk$, we have
\begin{equation*}
C_{n}^{0;\bfk;\ell}<0 \qquad\text{and}\qquad C_{n}^{j;\bfk;\ell}>0 \quad \forall n\gg 0.
\end{equation*}
Moreover, if $j,\ell\ge 1$ and $\bfk$ is admissible then
\begin{equation*}
 C_{n}^{j;\bfk;\ell}>\binom{n+j-1}{n}\frac{1}{\zeta(\bfk)^{\ell}} \quad \forall n\gg 0.
\end{equation*}
Hence, the Maclaurin series
\begin{equation*}
 -\left(\frac{x^{\dep(\bfk)}}{\Li_\bfk(x)}\right)^\ell  \quad\text{and}\quad
\frac{1}{(1-x)^j}\cdot \left(\frac{x^{\dep(\bfk)}}{\Li_\bfk(x)}\right)^\ell
\end{equation*}
both have \emph{eventually} positive coefficients.
\end{con}

The special case of multiple logarithm (i.e., all components of $\bfk$ are equal to 1) follows easily from
Theorem~\ref{thm:GregoryGeneralOrderA} and Theorem~\ref{thm:NorlundType} (see Corollary~\ref{cor:multilogGregory} and Corollary~\ref{cor:multilogNorlund}). On the other hand, we will prove a more precise version of Conjecture \ref{conj:Trivar} in  the special cases of polylogarithms and double polylogarithms in Theorem~\ref{thm:TrivarDepth1} and Theorem~\ref{thm:TrivarDepth2}, respectively. As a by product, for each zeta value $\zeta(k)$ (resp. double zeta value $\zeta(a,b)$), we find an infinite family of identities expressing its reciprocal as a sum of a rational number $\nu_{n,k}$ (resp. $\nu_{n,a,b}$) and an improper integral such that $1/\zeta(k)=\lim_{n\to\infty} \nu_{n,k}$ (resp. $1/\zeta(a,b)=\lim_{n\to\infty} \nu_{n,a,b}$). As far as we know, this is the first time the reciprocals of these values are approximated by rational numbers nontrivially.

\section{Transformation to integrals}\label{sec:contour}
In this section, we will first express $G_n^{(\ell)}>0$ in terms of some improper integrals by contour integration using complex analysis, motivated by an idea first appeared in \cite{Schr1880}.

To begin with, we set $z=-x-1$ in \eqref{defn-Gregorycoefficientsp} and get
\begin{align*}
1+\sum_{n=1}^\infty (-1)^n G_n^{(\ell)}(z+1)^n=(-1)^\ell \left(\frac{z+1}{\log(z)-\pi i}\right)^\ell, \quad |z+1|<1.
\end{align*}
Consider the path along $C$ as given in Fig. \ref{fig:C}.
\begin{figure}[h]
\begin{center}
\begin{tikzpicture}[scale=0.6]
\node (A) at (-1,0) {};
\node (B) at (2.6,0){};
\node (C) at (0,2.4){};
\draw[->] (-2.1,0) -- (-0.7,0) node[below]  {$C_\eps$} -- (B) node[below] {$x$} ;
\draw[->] (0,-2.1) -- (C) node[left] {$y$} ;
\draw [thick]  (0.22,0.12) arc(20:330:0.3);
\draw [->,thick]  (0.21,0.12)  -- (1,0.12) ;
\draw [thick]  (1,0.12)  -- (2,0.12) ;
\draw [thick]  (2,0.12) arc(4.45:357.55:2);
\draw  [thick]  (1,-0.12)  -- (0.21,-0.12) ;
\draw [->,thick]  (2,-0.12)  -- (1,-0.12) ;
\draw  (1.3,1.6) node[right]  {$C_R$} ;
\draw [->]  (1.402,1.402)  -- (1.401,1.403) ;
\draw [->]  (-0.254,0.253)  -- (-0.253,0.254) ;
\end{tikzpicture}
\end{center}
\caption{Integration contour C with $\eps\to 0^+$ and $R\to \infty$.}
\label{fig:C}
\end{figure}

For any $n>1$ we have
\begin{align*}
(-1)^{n+\ell} G_n^{(\ell)}= &\, \Res_{z=-1} \frac{1}{(z+1)^{n+1}} \left(\frac{z+1}{\log(z)-\pi i}\right)^\ell \\
=&\, \frac{1}{2\pi i} \int  \left(\frac{z+1}{\log(z)-\pi i}\right)^\ell \frac{dz}{(z+1)^{n+1}}\\
=&\, \frac{1}{2\pi i}\left( \int_\eps^R +\int_{C_R}+\int_R^\eps+ \int_{C_\eps} \right) \left(\frac{z+1}{\log(z)-\pi i}\right)^\ell \frac{dz}{(z+1)^{n+1}}.
\end{align*}

Taking $\eps\to 0$ and $R\to \infty$, we obtain
\begin{align*}
(-1)^{n+\ell}  G_n^{(\ell)}=&\, \frac{1}{2\pi i}  \int_0^\infty  \left(\frac{1}{(\log(z)-\pi i)^\ell}- \frac{1}{(\log(z)+\pi i)^\ell}\right) \frac{dz}{(z+1)^{n-\ell+1}}\\
=&\,    \frac{1}{2\pi i} \int_0^\infty  \frac{1}{(\log^2(z)+\pi^2)^\ell} \left( \sum_{j=0}^\ell \binom{\ell}{j}  \log^{\ell-j}(z)(\pi i)^j(1-(-1)^j) \right)\frac{dz}{(z+1)^{n-\ell+1}}\\
=&\,    \int_0^\infty  \frac{1}{(\log^2(z)+\pi^2)^\ell} \left( \sum_{k=1}^{\lceil \ell/2\rceil} (-1)^{k-1} \binom{\ell}{2k-1}
                                            \log^{\ell-2k+1}(z) \pi^{2k-2}  \right)\frac{dz}{(z+1)^{n-\ell+1}}\\
=&\,    \left( \int_1^0-\int_1^\infty \right)   \left( \sum_{k=1}^{\lceil \ell/2\rceil} (-1)^k \binom{\ell}{2k-1}
                                         \frac{\log^{\ell-2k+1}(z) \pi^{2k-2}}{(\log^2(z)+\pi^2)^\ell} \right)\frac{dz}{(z+1)^{n-\ell+1}}.
\end{align*}
Taking $z\to 1/z$ in the integral over $(0,1)$ and then combining the two integrals, we get
\begin{align} \label{equ:GnEllIntegral}
(-1)^{n+\ell}  G_n^{(\ell)}
=&\,     \int_1^\infty   \left( \sum_{k=1}^{\lceil \ell/2\rceil} (-1)^k \binom{\ell}{2k-1}
               \frac{\log^{\ell-2k+1}(z) \pi^{2k-2}}{(\log^2(z)+\pi^2)^\ell}  \right)\frac{((-1)^{\ell} z^{n-\ell-1}-1)dz}{(z+1)^{n-\ell+1}}
\end{align}
for all $n>\ell+1$, which has an easy corollary as follows.

\begin{cor}\label{cor:order2Conj}
We have $G_1^{(2)}=1, G_2^{(2)}=1/12, G_3^{(2)}=0$, and $(-1)^{n} G_n^{(2)}<0$ for all $n\ge 4$.
\end{cor}

\begin{proof}
The first two coefficients are easy to obtain. Assume $n\ge 3$.  Then \eqref{equ:GnEllIntegral} yields that
\begin{align*}
(-1)^n G_n^{(2)}=&\,    \int_1^\infty  \frac{-2\log(z)(z^{n-3}-1)}{(\log^2(z)+\pi^2)^2} \frac{dz}{(z+1)^{n-1}}<0
\end{align*}
for all $n>3$ and $G_3^{(2)}=0$. Thus, $(-1)^{n-1} G_n^{(2)}>0$ for all $n\ge 4$.
\end{proof}

\begin{re}\label{rem:nonZero}
When $\ell=1$, \eqref{equ:GnEllIntegral} gives the important integral expression of the classical Gregory coefficients
\begin{equation}\label{equ:nonZero}
(-1)^{n-1} G_n=\int_1^\infty   \frac{1}{\log^2(z)+\pi^2} \frac{z^{n-2}+1}{(z+1)^n}\, dz
\end{equation}
first obtained by Schr\"oder \cite{Schr1880} (also see \cite{Bl2017}). This is the original motivation for the current research.
\end{re}

\section{Vanishing of higher order Gregory coefficients}
Unlike the classical Gregory coefficients which are all nonzero because of \eqref{equ:nonZero}, we show in this section that
some particular higher order Gregory coefficient always vanishes, for each order $\ell\ge 2$.

\begin{thm}\label{thm:vanishing}
We have  $G_{\ell+1}^{(\ell)}=0$ for all even  $\ell$ and $G_{\ell}^{(\ell)}=0$ for all odd $\ell\ge 3$.
\end{thm}

\begin{proof}
If $\ell$ is even, then $G_{\ell+1}^{(\ell)}=0$ follows immediately from \eqref{equ:GnEllIntegral}.

Now we assume $\ell$ is odd and set $\ell=2m-1$. Then by \eqref{equ:GnEllIntegral}
\begin{align*}
G_{\ell}^{(\ell)}=&\,    \int_1^\infty    \left( \sum_{k=1}^m (-1)^k \binom{2m-1}{2k-1}
                                            \log^{2m-2k}(z) \pi^{2k-2}  \right)\frac{-dz}{z(\log^2(z)+\pi^2)^\ell}\\
=&\,   \int_0^\infty  \left( \sum_{k=1}^m (-1)^k \binom{2m-1}{2k-1}
           t^{2m-2k} \pi^{2k-2}  \right)  \frac{ - dt}{(t^2+\pi^2)^\ell}  \\
=&\,  \frac{-1}{\pi^{2m-1}}   \int_0^\infty  \sum_{k=1}^m (-1)^k \binom{2m-1}{2k-1}
                                             \frac{t^{2m-2k} \, dt}{(t^2+1)^{2m-1}}   \\
=&\,  \frac{-1}{\pi^{2m-1}}   \int_0^{\pi/2}  \sum_{k=1}^m (-1)^k \binom{2m-1}{2k-1}
                                           \frac{ \tan^{2m-2k}\theta  \,  d\theta }{\sec^{4m-4}\theta}   \\
=&\,  \frac{-1}{\pi^{2m-1}}  \sum_{k=1}^m (-1)^k \binom{2m-1}{2k-1}  \int_0^{\pi/2}  \cos^{2m+2k-4}\theta
                                          \sin^{2m-2k}\theta  \, d\theta  \\
=&\,  \frac{-1}{2\pi^{2m-1}}  \sum_{k=1}^m (-1)^k \binom{2m-1}{2k-1}  \Gamma\Big(\frac{2m+2k-3}2\Big)\Gamma\Big(\frac{2m-2k+1}2\Big)/ \Gamma\Big(\frac{4m-4}2+1\Big)  \\
=&\,  \frac{-1}{2\pi^{2m-1}}  \sum_{k=1}^m (-1)^k \binom{2m-1}{2k-1} \frac{(2m+2k-4)!(2m-2k)!}{4^{2m-2} (m+k-2)!(m-k)!(2m-2)!} \\
=&\,  \frac{-(2m-1)}{2^{4m-3}\pi^{2m-2}}  \sum_{k=1}^m  \frac{ (-1)^k (2m+2k-4)!}{ (2k-1)!(m+k-2)!(m-k)! }.
\end{align*}
By WZ method, if we set
\begin{equation*}
     g(k,m)=(2k^2 - 3k + 1)\cdot \frac{(-1)^k(2m + 2k - 4)!}{(2k - 1)!(m + k - 2)!(m - k)!},
\end{equation*}
then for all $1\le k\le m$, we have
\begin{equation*}
     g(k+1,m)-g(k,m)=-(2m^2 - 3m + 1)\cdot  \frac{(-1)^k (2m+2k-4)!}{ (2k-1)!(m+k-2)!(m-k)! }
\end{equation*}
where $g(m+1,m)=0$. Therefore, if $m\ge 2$ then
\begin{align*}
 G_{\ell}^{(\ell)}=&\,  \frac{1}{2^{4m-3}\pi^{2m-2}(m-1)}  \sum_{k=1}^m  \Big(g(k+1,m)-g(k,m)\Big)=0
\end{align*}
since $g(1,m)=0$. This completes the proof of the theorem.
\end{proof}

\section{A key lemma concerning the roots of some polynomials}

To prove Theorem~\ref{thm:GregoryGeneralOrder}, we need the following lemma.

\begin{lem}\label{equ:keyLem}
For all positive integer $\ell$, define the polynomial
\begin{equation*}
    H_{\ell}(x)= \sum_{k=1}^{\lceil \ell/2\rceil} (-1)^{k-1} \binom{\ell}{2k-1} x^{\ell-2k+1}.
\end{equation*}
Then we have
\begin{enumerate}
  \item[\upshape{(i)}] $H_{\ell}(x)$ has exactly $\ell-1$ simple real roots. Among these, $\lfloor (\ell-1)/2 \rfloor$ are positive roots.

  \item[\upshape{(ii)}] Let $\rho_\ell$ be the largest root of $H_{\ell}(x)$. Then the sequence $\{\rho_\ell\}_{\ell\ge 3}$ is strictly
  increasing starting with $\rho_3=\sqrt{3}/3$.
\end{enumerate}

\end{lem}

\begin{proof}
(i) will be proved by Sturm's theorem (see, e.g., \cite{Dorrie1965}). Set $r_0(x):=H_{\ell}(x)$ and for all $j\ge 1$,
\begin{equation*}
    r_{j}(x)= \sum_{k\ge 1} (-1)^{k-1}  \binom{\ell+j-2}{2k+2j-3} \binom{k+j-2}{j-1} x^{\ell-2k-j+1} .
\end{equation*}

We claim that for all $j\ge 1$ there are \emph{positive} numbers $b_j=b_j(\ell)$ and $c_j=c_j(\ell)$ such that
\begin{equation}\label{equ:Sturmj}
  x r_{j}(x) -b_j r_{j-1}(x) =c_j r_{j+1}(x).
\end{equation}

We now prove claim \eqref{equ:Sturmj} by induction. When $j=1$, we get
\begin{align*}
\ell x\cdot r_1(x) -(\ell -1)  r_0(x)
=&\,\sum_{k\ge 1} (-1)^{k-1}  x^{\ell-2k+1} \left[(\ell-2k+1)\binom{\ell}{2k-1}-(\ell -1)\binom{\ell}{2k-1}\right]\\
=&\,\sum_{k\ge 1} (-1)^{k}  x^{\ell-2k+1}  \binom{\ell}{2k-1}(2k-2) \\
=&\,\sum_{k\ge 1} (-1)^{k-1}  x^{\ell-2k-1} \binom{\ell}{2k+1}2k =2 r_2(x)
\end{align*}
by shifting the index $k\to k+1$. Thus $b_1=1 -1/\ell$, $c_j=2/\ell$, and claim \eqref{equ:Sturmj} holds for $j=1$.
In general, for all $j\ge 2$ we have
\begin{align*}
&\,2(2j-1) x\cdot r_{j}(x) - \frac{(\ell+j-2)(\ell-j)}{j-1}  r_{j-1}(x)\\
=&\, \sum_{k\ge 1}  (-1)^{k} x^{\ell-2k-j+2} \left[  \frac{(\ell+j-2)(\ell-j)}{j-1}  \binom{\ell+j-3}{2k+2j-5} \binom{k+j-3}{j-2} \right.\\
&\, \hskip5cm \left.-2(2j-1)\binom{\ell+j-2}{2k+2j-3} \binom{k+j-2}{j-1}\right]\\
=&\, 2\sum_{k\ge 1}  \frac{(-1)^{k} x^{\ell-2k-j+2}}{\ell+j-1}\binom{\ell+j-1}{2k+2j-3} \binom{k+j-2}{j-1}\\
&\, \hskip4cm  \times \Big[ (\ell-j)(2k+2j-3)-(\ell-2k-j+2)(2j-1) \Big]\\
=&\, 4j \sum_{k\ge 1}  (-1)^{k-1} x^{\ell-2k-j} \binom{\ell+j-1}{2k+2j-1} \binom{k+j-1}{j} \\
=&\, 4j r_{j+1}(x).
\end{align*}
Hence,
\begin{equation*}
 b_j= \frac{(l+j-2)(l-j)}{2(2j-1)(j-1)}>0, \quad c_j= \frac{4j}{2(2j-1)}>0
\end{equation*}
and therefore claim \eqref{equ:Sturmj} is proved.

Next, we set $R_0(x)=H_{\ell}(x)$, $R_1(x)=H_{\ell}'(x)$ and for all $j\ge 1$
\begin{equation*}
     R_{j+1}(x)=-\text{Remainder}\big(R_j(x),R_{j-1}(x)\big)=(q_j x+p_j)\cdot R_j(x)-R_{j-1}(x)
\end{equation*}
as defined by the Sturm algorithm, where $q_j$ and $p_j$ are some rational numbers since this process
is essentially a variant Euclidean algorithm in ${\mathbb Q}[x]$.
We now show by induction again that $p_j=0$ for all $j\ge 1$
and there are positive numbers $a_j=a_j (\ell)$ such that
\begin{equation}\label{equ:Rjx}
     R_j(x)=a_j r_j(x)
\end{equation}
for all $j\ge 0$. It holds clearly for $j=0,1$ by taking $a_0=1$ and $a_1=\ell$ since
\begin{align*}
  R_1(x)=H_{\ell}'(x)=&\, \sum_{k=1}^\ell (-1)^{k-1} \binom{\ell}{2k-1}(\ell-2k+1) x^{\ell-2k}\\
  = &\,  \sum_{k=1}^\ell (-1)^{k-1} \ell \binom{\ell-1}{2k-1} x^{\ell-2k}= \ell r_1(x).
\end{align*}
In general, assume $R_{i}=a_i r_i$ already holds for all $1\le i\le j$. Then
\begin{align*}
     R_{j+1}(x)=&\, (q_j x+p_j)\cdot R_j(x)-R_{j-1}(x) \\
     =&\, q_ja_j \big(b_j r_{j-1}(x) +c_j r_{j+1}(x)\big)+p_j a_j r_j(x)-a_{j-1} r_{j-1}(x) \\
     =&\, (q_ja_jb_j-a_{j-1}) r_{j-1}(x)+p_j a_j r_j(x)+q_ja_j c_j r_{j+1}(x).
\end{align*}
By definition and induction, $\deg R_{j+1}<\deg R_j=\deg r_j=\ell-j-1$ and $\deg r_{j+1},\deg r_j<\deg r_{j-1}$.
This yields immediately $q_ja_jb_j-a_{j-1}=0$ and then $p_j=0$.
Therefore $a_{j+1}=q_ja_j c_j  =a_{j-1}c_j/b_j>0$.

From \eqref{equ:Sturmj} and  \eqref{equ:Rjx}, all the leading coefficients of $R_j(x)$ are positive.
This implies that the sequence
$\{R_j(+\infty)\}_{0\le j\le \ell-1}$ is positive and $\{R_j(-\infty)\}_{0\le j\le \ell-1}$ is alternating.
Moreover, $\gcd(H_{\ell}(x),H_{\ell}'(x))=1$ by the Euclidean algorithm since $r_{\ell-1}(x)=1$ which shows that all roots of $H_{\ell}(x)$ are simple.
Therefore, Sturm's theorem tells us that $H_{\ell}(x)$ has $\ell-1=\deg H_{\ell}$ real roots.
Finally, the last sentence in the lemma holds because $H_{\ell}(x)$ is either an even or odd function depending on whether $\ell$ is odd or even.

(ii) From the proof of (i) we see that all the roots $\{\sigma_1<\cdots<\sigma_{\ell-2}\}$  of $H_{\ell}'(x)=lH_{\ell-1}(x)$ are critical values of $H_\ell(x)$ which has exactly $\ell-1$ roots. Thus $ \sigma_1,\ldots,\sigma_{\ell-2}$ are alternating in sign with the largest value $\sigma_{\ell-2}=\rho_{\ell-1}$ satisfying $H_\ell(\rho_{\ell-1})<0$ since $H_\ell(+\infty)=+\infty$. Therefore, $\rho_\ell>\rho_{\ell-1}$. Further, $H_3(x)=3x^2-1$ and therefore $\rho_3=\sqrt{3}/3$.
\end{proof}

\section{Higher order Gregory coefficients}\label{sec:GregoryGeneralOrder}

We are now ready to prove the following theorem which states that higher order Gregory coefficients are eventually alternating for each order $\ell\ge 1$.

\begin{thm}\label{thm:GregoryGeneralOrder}
For each order $\ell\ge 1$, there is some positive integer $\Gr(\ell)$ such that
\begin{equation*}
(-1)^{n-1} G_n^{(\ell)}>0
\end{equation*}
for all $n\ge \Gr(\ell)$.
\end{thm}

\begin{proof} The case $\ell=1$ is well-known and the case $\ell=2$ is already confirmed by Corollary~\ref{cor:order2Conj} explicitly. We now assume $\ell\ge 3$ and set
\begin{align*}
h(z)=h^{(\ell)}(z):= &\, (-1)^{\ell} \left( \sum_{k=1}^{\lceil \ell/2\rceil} (-1)^k \binom{\ell}{2k-1}
               \log^{\ell-2k+1}(z) \pi^{2k-2}  \right)\cdot  (\log^2(z)+\pi^2)^{-\ell} \\
=&\, (-1)^{\ell-1} \left(\ell \log^{\ell-1}(z)-\binom{\ell}{3} \log^{\ell-3}(z)\pi^2+ \cdots\right) \cdot  (\log^2(z)+\pi^2)^{-\ell}
\end{align*}
and
\begin{equation*}
f_n(z)=f_n^{(\ell)}(z):=\frac{1-(-1)^{\ell} z^{n-\ell-1}}{(z+1)^{n-\ell+1}}.
\end{equation*}
Let $H_\ell(x)$ be the polynomial as defined in Lemma~\ref{equ:keyLem}. Then
\begin{equation*}
h(z)=(-\pi)^{\ell-1}H_{\ell}\big(\log(z)/\pi\big) \cdot (\log^2(z)+\pi^2)^{-\ell}
\end{equation*}
and therefore its largest real root $r=r^{(\ell)}$ satisfies
\begin{equation*}
r=\exp(\pi \rho_\ell)\ge \exp(\pi \rho_3)=\exp(\pi\sqrt{3}/3)\approx 6.13.
\end{equation*}
Hence, $(-1)^{\ell-1} h(z) = h(1/z) >0$ for all $z>r$. By \eqref{equ:GnEllIntegral} we see that for all $n>4$
\begin{equation}\label{equ:2Integrals}
(-1)^{n-1} G_n^{(\ell)}
= \int_1^\infty h(z) f_n(z)\, dz
= \int_1^{r} h(z)  f_n(z)\, dz+\int_0^{1/r} h(z)  f_n(z)\, dz
\end{equation}
by the substitution $z\to 1/z$ in the second integral.

We will proceed to prove that the second integral on the right-hand side of \eqref{equ:2Integrals},
which is always positive, dominates the first. The second integral is easy to estimate
but the first must be cut into two parts and then bounded separately.

We now consider the second integral first. Set
\begin{equation}\label{equ:alpha}
\alpha=\alpha^{(\ell)}:=\int_0^{1/r}  h(z) \, dz>0.
\end{equation}
Note that
\begin{equation*}
     f_n'(z):=\frac{ (-1)^{\ell-1}z^{n-\ell-2}(n-\ell-1-2z)-n+\ell-1}{(z+1)^{n-\ell+2}}.
\end{equation*}
Hence, $f_n'(z)<0$ for all $n>\ell+3$. Indeed, this is obvious if $\ell$ is even. If $\ell$ is odd, then
\begin{equation*}
    z^{n-\ell-2}(n-\ell-1-2z)-n+\ell-1<(n-\ell-1-2z)-n+\ell-1=-2-2z<0.
\end{equation*}
Thus, we get
\begin{equation}\label{equ:2ndIntegralBd}
 \int_0^{1/r} h(z)  f_n(z)  \, dz > f_n(1/r) \int_0^{1/r}  h(z)  \, dz =\alpha f_n(1/r)>0.
\end{equation}

Now we turn to the first integral on the right-hand side of \eqref{equ:2Integrals}.
We have to consider two separate cases depending on the parity of $\ell$.

\medskip\noindent
(i) If $\ell$ is odd then clearly for all $z\in [1, r]$
\begin{equation*}
0< f_n(z)=\frac{z^{n-\ell-1}+1}{(z+1)^{n-\ell+1}}<1
\end{equation*}
and
\begin{equation*}
f_n'(z)=\frac{g_n(z)}{(z+1)^{n-\ell+2}}, \quad\text{where}\quad g_n(z)=z^{n-\ell-2}(n-\ell-1-2z)-n+\ell-1.
\end{equation*}
Since
\begin{equation*}
g_n'(z)=(n-\ell-1)z^{n-\ell-3}(n-\ell-2-2z)>0
\end{equation*}
for all $1\le z\le r< (n-\ell-2)/2$ if $n$ satisfies
\begin{equation}\label{equ:keyRestrict}
n\ge \ell+2r+2,
\end{equation}
$f_n'(z)$ is increasing and has exactly one zero $z_0\in [1,r]$.
We see that $f_n(z)$ attains the minimum at $z_0$ (as $f_n'(1)=-2^{\ell-n}<0$) and $f_n(z)$ achieves maximum at either 1 or $r$.
Moreover, we must have $z_0<2$ since
\begin{equation*}
g_n(2)> 2(n-\ell-5)-n+\ell-1=n-\ell-11>0
\end{equation*}
by \eqref{equ:keyRestrict} as $r>6$. This implies that $f_n(z)$ is increasing on $[2,r]$.
Further, it can be proved easily that
\begin{equation*}
\left(\frac{3}{4}\right)^{n-\ell+1}< \frac14
\end{equation*}
for all $n$ satisfying \eqref{equ:keyRestrict}. Hence,
\begin{equation}\label{equ:fbLargest}
f_n(1)=\frac{1}{2^{n-\ell+1}}<\frac{2^{n-\ell-1}}{3^{n-\ell+1}}<\frac{2^{n-\ell-1}+1}{3^{n-\ell+1}}=f_n(2)<f_n(b)<f_n(r)
\end{equation}
for all $b\in (2,r)$.

We now recall that $h(r)=0$ which implies the existence of some $b\in (2,r)$ such that
\begin{equation*}
\beta_2 (b):= \int_b^r | h(z)| \, dz < \alpha r^2,
\end{equation*}
where $\alpha$ is defined by \eqref{equ:alpha}, which is equivalent to  $\alpha>\beta_2 (b)/r^2$.
Thus,
\begin{align}
\frac {f_n(b)}{f_n(1/r)}=&\, \frac{b^{n-\ell-1}+1}{(b+1)^{n-\ell+1}}\cdot\frac{ (1/r+1)^{n-\ell+1}}{(1/r)^{n-\ell-1}+1}\nonumber\\
=&\, \frac{1}{b^2}  \cdot \frac{(1/b)^{n-\ell-1}+1}{(1/r)^{n-\ell-1}+1} \left(\frac{1/r+1}{1/b+1}\right)^{n-\ell+1}\nonumber\\
<&\,  \frac{2}{b^2}\cdot \left(\frac{1/r+1}{1/b+1}\right)^{n-\ell+1}\nonumber\\
<&\,  \frac{1}{\beta_1 (b)}\left(\alpha-\frac{\beta_2 (b) }{r^2}\right)   \label{equ:geomDecrease}
\end{align}
for all $n\gg 0$ since $1/r+1<1/b+1$, where
\begin{equation*}
\beta_1 (b):= \int_1^b | h(z)| \, dz.
\end{equation*}
By \eqref{equ:fbLargest}, for all sufficiently large $n$, we get
\begin{align*}
\left| \int_1^r h(z) f_n(z)\, dz \right|
< &\,    f_n(b)\beta_1 (b) +  f_n(r)\beta_2 (b) \\
< &\, f_n(1/r) \left(\alpha-\frac{\beta_2 (b)}{r^2} \right)+  f_n(1/r) \frac{\beta_2 (b)}{r^2}
=\alpha f_n(1/r)
\end{align*}
by \eqref{equ:geomDecrease} and the fact that $r^2f_n(r) =f_n(1/r)$.
Combined with \eqref{equ:2Integrals} and \eqref{equ:2ndIntegralBd}, this yields that
\begin{equation}\label{equ:goal}
(-1)^{n-1} G_n^{(\ell)} \ge \alpha f_n(1/r)- \left| \int_1^r h(z) f_n(z)\, dz \right|>0.
\end{equation}

\medskip\noindent
(ii) If $\ell$ is even then  for all $z> 1$
\begin{equation*}
0<F_n(z)=-f_n(z)=\frac{z^{n-\ell-1}-1}{(z+1)^{n-\ell+1}}
<\frac{z^{n-\ell-1}}{(z+1)^{n-\ell+1}}
\le \left(\frac{z}{z+1} \right)^{n-\ell+1}<1.
\end{equation*}
We now show that $F_n(z)$ is increasing on $(1,r]$.
To do this, we notice  that
\begin{equation*}
F_n'(z)=\frac{g_n(z)}{(z+1)^{n-\ell+2}}=\frac{z^{n-\ell-2}(n-\ell-1-2z)+n-\ell+1}{(z+1)^{n-\ell+2}}
\end{equation*}
and
\begin{equation*}
g_n'(z)=(n-\ell-1)z^{n-\ell-3}(n-\ell-2-2z)>0
\end{equation*}
for all $1\le z\le r$ if $n$ satisfies \eqref{equ:keyRestrict}.
Hence, the function $g_n(z)$ is increasing on $[1,r]$. This implies that $g_n(z)>g_n(1)=2n-2\ell-2>0$.
Thus, $F_n'(z)>0$ and therefore $F_n(z)$ is increasing on $[1,r]$ with $F_n(1)=0$.
Hence, by applying the same argument as in  case (i) for odd $\ell$,
after replacing $f_n(z)$ by $F_n(z)$, we may find $2<b<r$ such that $\beta_2 (b)< \alpha r^2$
and
\begin{align*}
\frac {F_n(b)}{f_n(1/r)}=&\, \frac{b^{n-\ell-1}-1}{(b+1)^{n-\ell+1}}\cdot\frac{ (1/r+1)^{n-\ell+1}}{1-(1/r)^{n-\ell-1}}\\
=&\, \frac{1}{b^2}\cdot \frac{1-(1/b)^{n-\ell-1}}{1-(1/r)^{n-\ell-1}} \left(\frac{1/r+1}{1/b+1}\right)^{n-\ell+1}\\
< &\,  \frac{1}{b^2}\cdot \left(\frac{1/r+1}{1/b+1}\right)^{n-\ell+1}\\
<&\, \frac{1}{\beta_1 (b)}\left(\alpha-\frac{\beta_2 (b) }{r^2}\right)
\end{align*}
for all $n\gg 0$ since $1/r+1<1/b+1$. Hence, by exactly the same argument as
in the odd case, we see that \eqref{equ:goal} holds for all sufficiently large $n$.

We have now completed the proof of Theorem~\ref{thm:GregoryGeneralOrder}.
\end{proof}

We can now obtain an interesting corollary of Theorem~\ref{thm:GregoryGeneralOrder} concerning multiple logarithms.
For this purpose, we need to recall an important expression of multiple polyloarithms in terms of iterated path integrals.

In a series of papers \cite{KTChen1954,KTChen1971,KTChen1977}, K. T. Chen systematically developed the theory of iterated path integrals. For any 1-forms $\ga_1,\dots,\ga_w$ on a complex manifold $M$ and a path ${\texttt p}:[0,1]\to M$, we set
\begin{equation*}
\int_{\texttt p} \ga_1,\dots,\ga_w :=\int_{1>t_1>\cdots>t_w>0} {\texttt p}^*(\ga_1)(t_1)\cdots   {\texttt p}^*(\ga_w)(t_w).
\end{equation*}
Using this, for any composition $\bfk=(k_1,\dots,k_d)$ and $z\in\CC\setminus [1,\infty)$, we define
\begin{equation*}
    \Li_\bfk(z)=\int_0^z \left(\frac{dt}{t}\right)^{k_1-1}\frac{dt}{1-t}\cdots \left(\frac{dt}{t}\right)^{k_1-1}\frac{dt}{1-t}
\end{equation*}
where the path from 0 to 1 lies entirely in $\CC\setminus [1,\infty)$. It is easy to see that when $|z|<1$, this definition agrees with that in \eqref{equ:Polylog}.

\begin{cor}\label{cor:multilogGregory}
Let $k\in\N$ and let $\{1\}^k$ denote the sequence of $1$'s repeating $k$ times. Then the Maclarin series
\begin{equation*}
    \frac{-x^k}{\Li_{\{1\}^k}(x)}
\end{equation*}
has eventually positive coefficients, namely, $C_n^{0;\{1\}^k;1}<0$ for all $n\gg 0$.
\end{cor}
\begin{proof}
By the shuffle relations of iterated integrals (see \cite[(1.5.1)]{KTChen1971} or \cite{Ree1958}), we get
\begin{equation*}
\Li_{\{1\}^k}(x)=\int_0^x \left(\frac{dt}{1-t}\right)^k=\frac{1}{k!} \left(\int_0^x \frac{dt}{1-t}\right)^k=\frac{\Li_1(x)^k}{k!}.
\end{equation*}
Hence,
\begin{align*}
\frac{-x^k}{\Li_{\{1\}^k}(x)}=&\,-k!\left(\frac{x}{\Li_1(x)}\right)^k
=k! \left(-1 +\sum_{n=1}^\infty (-1)^{n-1} G_n^{(\ell)}x^n\right).
\end{align*}
The corollary follows from Theorem~\ref{thm:GregoryGeneralOrder} immediately.
\end{proof}

\section{Bounds on eventual positiveness of higher order Gregory coefficients}

Notice that the proof of  Theorem~\ref{thm:GregoryGeneralOrder} in Section~\ref{sec:GregoryGeneralOrder} is effective and there is some room to choose the value $b$ to cut the second integral on the right-hand side of \eqref{equ:2Integrals}. Using this idea, by Mathematica computation, we are able to find the upper bound $n_1(\ell)$ for $\Gr(\ell)$ for $\ell\le 10$ with the optimal choice of $b$ in Table~\ref{tbl:n0bound}.
\begin{table}[h]
\centering\begin{tabular}{|c|c|c|c|c|c|c|c|c|}
\hline
$\ell$ &  3 & 4 & 5 & 6 & 7 & 8 & 9 & 10   \\
\hline
$r\approx$ &
6.13 & 23.14 & 75.49 & 230.76 & 681.01  & 1967 & 5065  & 15817    \\
\hline
$b$ & 2.6 & 11.3 & 38 & 116 & 355 &1028 & 2925 & 8228   \\
\hline
$n_1(\ell)$ &
26 & 122 & 467 & 1627 & 5365 & 17195 & 54155 & 169344    \\
\hline
\end{tabular}
\caption{Upper bound $n_1(\ell)\ge \Gr(\ell)$.}
\label{tbl:n0bound}
\end{table}

However, these bounds are far from optimal. In reality, for even $\ell$, we actually can obtain much better bounds by searching for $n_2(\ell)$
such that for all $n\ge n_2(\ell)$
\begin{equation*}
\alpha^{(\ell)} -  \int_1^{r} |h(z)|  \frac{1}{z^2}\left(\frac{1/r+1}{1/z+1}\right)^{n-\ell+1} \, dz  >0
\end{equation*}
since this expression is clearly increasing as $n$ increases. Therefore,
\begin{align*}
(-1)^{n-1} G_n^{(\ell)}=&\, \int_1^{r} h(z)  f_n(z)\, dz+\int_0^{1/r} h(z)  f_n(z)\, dz\\
>&\,   f_n(1/r) \int_0^{1/r} |h(z)| \, dz - \left| \int_1^r h(z) f_n(z)\, dz \right|  \\
= &\,f_n(1/r)\left(\alpha^{(\ell)}- \int_1^{r} |h(z)| \frac{1- (1/z)^{n-\ell-1}}{z^2((1/r)^{n-\ell-1}+1)}
    \left(\frac{1/r+1}{1/z+1}\right)^{n-\ell+1}  \, dz \right)\\
>&\, f_n(1/r)\left(\alpha^{(\ell)}- \int_1^{r} |h(z)| \frac{1}{z^2}
    \left(\frac{1/r+1}{1/z+1}\right)^{n-\ell+1}  \, dz \right)>0
\end{align*}
for all $n\ge n_2(\ell)$.

If $\ell$ is odd we can modify the above slightly by breaking $[1,r]$ into two parts.
Namely, we find the smallest $n$ (say $n_2(\ell)$) such that
\begin{equation*}
\alpha^{(\ell)} -  \int_1^2 |h(z)|  \frac{2}{z^2}\left(\frac{1/r+1}{1/z+1}\right)^{n-\ell+1} \, dz
-\int_2^{r} |h(z)|  \frac{1+(1/2)^{n-\ell-1}}{z^2}\left(\frac{1/r+1}{1/z+1}\right)^{n-\ell+1} \, dz  >0.
\end{equation*}
Then \eqref{equ:goal} holds for all $n\ge n_2(\ell)$.

With this improved algorithm, we find the better bounds $n_2(\ell)$ of $\Gr(\ell)$ in Table~\ref{tbl:Newn0bound}.
\begin{table}[h]
\centering\begin{tabular}{|c|c|c|c|c|c|c|c|c|}
\hline
$\ell$ &  3 & 4 & 5 & 6 & 7 & 8 & 9 & 10   \\
\hline
$n_2(\ell)$ &
16 & 52 & 168 & 519 & 1548 & 4514 & 12944 & 36685    \\
\hline
\end{tabular}
\caption{Improved upper bound $n_2(\ell)\ge \Gr(\ell)$.}
\label{tbl:Newn0bound}
\end{table}

In particular, we found the exact values $\Gr(3)=11$, $\Gr(4)=36$, $\Gr(5)=113$, $\Gr(6)=346$, and $\Gr(7)=1030$ by an easy computer search using the new bound $n_2(\ell)$. We can see that the bound in Table~\ref{tbl:n0bound} is not optimal yet, but should be no more than twice of the true value of $\Gr(\ell)$. It would be interesting to find an even better bound in general with perhaps some new ideas so that the precise value $\Gr(\ell)$ can be determined. In particular, we would like to know if the following conjecture always holds.

\begin{con}
For all $\ell\ge 5$ we have $0.47\cdot 3^\ell<\Gr(\ell)<0.477\cdot 3^\ell.$
\end{con}

We also notice that the sequence $\{r^{(\ell)}\}_{\ell\ge 2}$ (of the largest real root of $H_{\ell}\big(\log(z)/\pi\big)$) seems to increase by roughly tripling as $\ell$ increases, similarly to $\{\Gr(\ell)\}_{\ell\ge 2}$. We wonder if there is any theoretical reason for this pattern. By contrast, we have the natural upper bound
\begin{equation*}
r^{(\ell)} \le A^{(\ell)}:=\exp\left(\pi \sqrt{\frac{(\ell-1)(\ell-2)}{6}}\right)
\end{equation*}
since for all odd $k< \lceil \ell/2\rceil$ we have
\begin{align*}
\ &\, \binom{\ell}{2k-1} \frac{(\ell-1)(\ell-1)}{6}-\binom{\ell}{2k+1} \\
= &\, \binom{\ell}{2k-1} \left( \frac{(\ell-1)(\ell-1)}{6}- \frac{(\ell-2k+1)(\ell-2k)}{(2k+1)(2k)} \right)\ge 0.
\end{align*}
However, $\{A^{(\ell)}\}_{\ell\ge 2}$ increases at a rate of about $\exp(\pi \sqrt{6}/6)\approx 3.606$ for all large $\ell$, likely a much faster rate than that of $\{r^{(\ell)}\}_{\ell\ge 2}$.

\section{Higher order N\"olund numbers of the first and second kind}\label{sec:Norlund1}

We now apply the same mechanism used in Section~\ref{sec:GregoryGeneralOrder} to prove some similar results for N\"olund numbers and their higher order generalizations.

\begin{thm}\label{thm:Norlund1}
For each order $\ell\ge 1$ and all $n\ge 1$, we have
\begin{equation*}
(-1)^{n} {\tilde N}_n^{(\ell)}>0.
\end{equation*}
Namely, the Maclaurin series of
\begin{equation*}
\left(\frac{x}{(1+x)\log(1+x)}\right)^\ell
\end{equation*}
has positive coefficients.
\end{thm}

\begin{proof}
For any $n\ge 1$ we have
\begin{align*}
(-1)^{n+\ell-1}  N_n^{(\ell)} = &\, \Res_{z=-1} \frac{1}{(z+1)^{n+1}}  \frac{(z+1)^\ell}{z(\log(z)-\pi i)^\ell} \\
=&\, \frac{1}{2\pi i} \int  \left(\frac{1}{\log(z)-\pi i}\right)^\ell \frac{dz}{z(z+1)^{n-\ell+1}}\\
=&\, \frac{1}{2\pi i}\left( \int_\eps^R +\int_{C_R}+\int_R^\eps+ \int_{C_\eps} \right) \left(\frac{1}{\log(z)-\pi i}\right)^\ell \frac{dz}{z(z+1)^{n-\ell+1}}
\end{align*}
where the contour here is the same as in Figure~\ref{fig:C}. Note that
\begin{equation*}
     \int_0^{2\pi}  \left(\frac{1}{\log(\varepsilon)+i\theta-\pi i}\right)^\ell \frac{i d\theta}{(\varepsilon e^{i\theta}+1)^{n-\ell+1}} \to 0 \quad\text{as}\quad \varepsilon \to 0.
\end{equation*}
Taking $\eps\to 0$ and $R\to \infty$, we obtain
\begin{align*}
(-1)^{n+\ell-1}  N_n^{(\ell)} =&\, \frac{1}{2\pi i}  \int_0^\infty  \left(\frac{1}{(\log(z)-\pi i)^\ell}- \frac{1}{(\log(z)+\pi i)^\ell}\right) \frac{dz}{z(z+1)^{n-\ell+1}}\\
=&\,    \frac{1}{2\pi i} \int_0^\infty  \frac{1}{(\log^2(z)+\pi^2)^\ell} \left( \sum_{j=0}^\ell \binom{\ell}{j}  \log^{\ell-j}(z)(\pi i)^j(1-(-1)^j) \right)\frac{dz}{z(z+1)^{n-\ell+1}}\\
=&\,    \int_0^\infty  \frac{1}{(\log^2(z)+\pi^2)^\ell} \left( \sum_{k=1}^{\lceil \ell/2\rceil} (-1)^{k-1} \binom{\ell}{2k-1}
                                            \log^{\ell-2k+1}(z) \pi^{2k-2}  \right)\frac{dz}{z(z+1)^{n-\ell+1}}\\
=&\,    \left( \int_1^0-\int_1^\infty \right)   \left( \sum_{k=1}^{\lceil \ell/2\rceil} (-1)^k \binom{\ell}{2k-1}
                                         \frac{\log^{\ell-2k+1}(z) \pi^{2k-2}}{(\log^2(z)+\pi^2)^\ell} \right)\frac{dz}{z(z+1)^{n-\ell+1}}.
\end{align*}
Taking $z\to 1/z$ in the integral over $(0,1)$ and then combining the two integrals, we get
\begin{align*}
(-1)^{n+\ell-1}  N_n^{(\ell)}
=&\,     \int_1^\infty   \left( \sum_{k=1}^{\lceil \ell/2\rceil} (-1)^k \binom{\ell}{2k-1}
               \frac{\log^{\ell-2k+1}(z) \pi^{2k-2}}{(\log^2(z)+\pi^2)^\ell}  \right)\frac{((-1)^{\ell} z^{n-\ell+1}-1)dz}{z(z+1)^{n-\ell+1}}.
\end{align*}
When $\ell=1$ we see that
\begin{align*}
(-1)^{n} N_n  = \int_1^\infty   \frac{( z^{n}+1)dz}{z(\log^2(z)+\pi^2)(z+1)^{n}} >0.
\end{align*}
This implies that all the coefficients of
\begin{equation*}
\frac{-x}{(1-x)\log(1-x)}=1+\sum_{n=1}^\infty (-1)^{n} N_n \frac{x^n}{n!}
\end{equation*}
are positive and therefore Theorem~\ref{thm:Norlund1} follows immediately since any positive power of a series with positive coefficient will still have positive coefficients.
\end{proof}

\begin{thm}\label{thm:Norlund2}
For each order $\ell\ge 1$ and all $n\ge 1$, there is a positive integer $\Nor(\ell)$ such that
\begin{equation*}
(-1)^{n} N_n^{(\ell)}>0
\end{equation*}
for all $n\ge \Nor(\ell)$.
\end{thm}

\begin{proof}
The proof of Theorem~\ref{thm:Norlund2} follows
a similar argument as used in Section~\ref{sec:GregoryGeneralOrder}. We thus leave the details to the interested reader.
\end{proof}

We list some of the first few values of $\Nor(\ell)$ in Table~\ref{tbl:NewNorBound}.

\begin{table}[h]
\centering\begin{tabular}{|c|c|c|c|c|c|c|c|c|}
\hline
$\ell$ & 1 & 2 &3 & 4 & 5 & 6 & 7 & 8   \\
\hline
$\Nor(\ell)$ &
1 & 2 & 4  & 11  &  36 &  113 &  346  & 1030    \\
\hline
\end{tabular}
\caption{Values of $\Nor(\ell)$ for $\ell\le 8$.}
\label{tbl:NewNorBound}
\end{table}

\section{More general higher order N\"olund-type numbers}\label{sec:NorlundGeneral}
In  this section we first prove some general preliminary results that can be applied to any power series with (eventually) positive coefficients. Then we prove in Theorem~\ref{thm:NorlundType} that the higher order N\"olund-type numbers $N_n^{j,\ell}$ defined by \eqref{defn-Norlund-type} are eventually alternating for any fixed $j,\ell\ge 1$.

\begin{lem}\label{lem:eventualPos}
Suppose a power series $f(x)=\sum_{n\ge 0} a_n x^n$ has positive (resp. eventually positive) coefficients and $f(1)$ diverges. Then
the power series
\begin{equation*}
g(x)=\frac{f(x)}{1-x}= \sum_{n\ge 0} \left(\sum_{j=0}^n a_j \right) x^n
\end{equation*}
also has positive (resp. eventually positive) coefficients.
\end{lem}

\begin{proof}
Assume $g(x)=\sum_{n\ge 0} b_n x^n$. Suppose $f(x)$ has eventually positive coefficients, namely, $a_n>0$ for all $n\ge n_0$. Then the sequence
\begin{equation}\label{equ:bn's}
 \bigg\{b_n \bigg\}_{n>n_0}=  \bigg\{\sum_{j=0}^n a_j \bigg\}_{n\ge n_0}
\end{equation}
is increasing and diverges to $\infty$ by the condition on $f(1)$.
Therefore, $g(x)$ has eventually positive coefficients and $g(1)$ diverges to $\infty$.

Moreover, if $f(x)$ has positive coefficients, namely, $n_0=0$ in the above, then $b_n>0$ for all $n\ge n_0$ by \eqref{equ:bn's}.
Hence, $g(x)$ has positive coefficients,
\end{proof}

\begin{thm} \label{thm:NorlundType}
For each order $\ell\ge 1$ and $j\ge 1$, we have
\begin{equation*}
(-1)^{n} N_n^{j,\ell}>0
\end{equation*}
for all $n\gg 0$. Namely, the Maclaurin series of
\begin{equation*}
\frac{1}{(1+x)^j} \left(\frac{x}{\log(1+x)}\right)^\ell
\end{equation*}
has eventually alternating coefficients. Moreover, for all $j\ge \ell\ge 1$, we have
\begin{equation*}
(-1)^{n} N_n^{j,\ell}>0
\end{equation*}
for all $n\ge 0$.
\end{thm}

We notice that the cases $j=0$ and $j=1$ of Theorem~\ref{thm:NorlundType} are equivalent to Theorem~\ref{thm:GregoryGeneralOrder} and Theorem~\ref{thm:Norlund2}, respectively. To prove the first part of Theorem~\ref{thm:NorlundType}, we apply induction on $j\ge 1$ to prove the following more exact and stronger statement. To prove the last part of Theorem~\ref{thm:NorlundType},  we only need induction on $j\ge \ell$ with the initial case $j=\ell$ given by  Theorem~\ref{thm:Norlund1}.

\begin{lem} \label{lem:fjInduction}
Fix $\ell\ge 0$. For any $j\ge 1$ define
 \begin{align*}
f_j(x):=\frac{1}{(1-x)^j}\cdot \left(\frac{-x}{\log(1-x)}\right)^\ell =\sum_{n=1}^\infty a_{j,n} x^n.
\end{align*}
Then $f_j(x)$ has eventually positive coefficients and $f_j(1)$ diverges to $\infty$.
\end{lem}

\begin{proof}
Theorem~\ref{thm:Norlund2} implies that $f_1(x)$ has eventually positive coefficients.
Further, from \cite[(16) and (19)]{Bl2016} we obtain that
\begin{equation*}
     a_{1,n}=(-1)^n N_n^{(\ell)} \sim \frac{1}{\log n}.
\end{equation*}
Therefore, $f_1(1)$ diverges to $\infty$.

For general $j$, if we assume $f_j(x)$ satisfies the conclusion of Lemma~\ref{lem:fjInduction} then
it satisfies the conditions of Lemma~\ref{lem:eventualPos}. Therefore
\begin{align*}
f_{j+1}(x):=\frac{1}{1-x}f_j(x)
\end{align*}
has eventually positive coefficients and $f_{j+1}(1)$ diverges to $\infty$ by Lemma~\ref{lem:eventualPos}.
This completes the proof of Lemma~\ref{lem:fjInduction}.
\end{proof}

\begin{cor}\label{cor:multilogNorlund}
For any integers $j,k\ge 1$, the Maclarin series
\begin{equation*}
    \frac{-x^k}{(1-x)^j \Li_{\{1\}^k}(x)}
\end{equation*}
has positive coefficients, namely, $C_n^{j;\{1\}^k;1}>0$ for all $n\ge 0$.
\end{cor}
\begin{proof}
As in the proof of Corollary~\ref{cor:multilogGregory}, by the shuffle relations of iterated integrals, we have
\begin{align*}
\frac{x^k}{(1-x)^j\Li_{\{1\}^k}(x)}=&\,\frac{k!}{(1-x)^j}\left(\frac{x}{\Li_1(x)}\right)^k.
\end{align*}
The corollary follows from Theorem~\ref{thm:NorlundType} immediately.
\end{proof}

\section{Polylogarithm and double polylogarithm cases of Conjecture~\ref{conj:Trivar}}\label{sec:TrivarDepth1}

We will confirm Conjecture~\ref{conj:Trivar} in the polylogarithm and double polylogarithm cases, i.e., when $\dep(\bfk)=1$ or 2.
The crucial step of our proof is a detailed analysis of behavior of the polylogarithm $\Li_\bfk(z)$ for complex variable $z$ through its iterated integral expression and then its behavior as $z$ approaches the positive real number line from top and bottom.

To compute the coefficients $C_{n}^{0;\bfk;1}$, we will apply the same technique as the one we used in Section~\ref{sec:GregoryGeneralOrder}. Therefore, we need to control the behavior of the polylogarithm $\Li_\bfk(z)$ as $z$ approaches the positive real number line from top and bottom. If $x>1$ then we define $\Li_\bfk(x\pm  0i)$ by the same iterated path integral as above except that the terminal end of the path at $x\pm 0i$ is considered as a point on the upper (resp. lower) half-plane for $x+ 0i$ (resp. $x-0i$).

We now recall the key result in \cite{Panzer2017}. It concerns with multiple polylogarithms of multi-variables whose specialization at $(z,1,\dots,1)$ yields the single variable multiple polylogarithms in this paper. In particular, for any admissible $\bfk$, i.e., $\bfk=(k_1,\dots,k_d)\in \N^d$ with $k_1>1$, the multiple zeta value $\zeta(\bfk):=\Li_\bfk(1)$.

\begin{lem} \label{lem:PanzerRed}
\emph{(cf. \cite[Theorem 1.2]{Panzer2017})}
For any $0\ne z\in\CC\setminus[0,\infty)$ and $\bfk=(k_1,\dots,k_d)\in \N^d$
\begin{align*}
 \Li_\bfk(z)=&\, (-1)^{|\bfk|+d} \Li_\bfk\Big(\frac1z\Big)
 + \sum_{
 \substack{ a,b\ge 0, \bfn: \ \text{admissible}\\
a+b+|\bfm|+|\bfn|=|\bfk|,\\
\dep(\bfm)+\dep(\bfn)< d
}} c(a,b,\bfm,\bfn) (2\pi i)^a \log^b(z) \Li_\bfm(z)\zeta(\bfn)
\end{align*}
where $c(a,b,\bfm,\bfn)\in\Q$.
\end{lem}

\begin{re}
The key observation for computing single variable multiple polylogarithms $\Li_\bfk(z)$ is that if $|z|\le 1$ then it can be computed by its series definition when $\bfk$ is admissible or $z\ne 1$. If $|z|>1$ (resp. $z>1$) then we may use Lemma~\ref{lem:PanzerRed} (resp. Lemma~\ref{lem:PanzerDepth1Red}) to compute.
\end{re}

\subsection{Depth one case: polylogarithms}

When the depth of $\bfk$ is 1, Lemma~\ref{lem:PanzerRed} above can be improved to an explicit form.
The result below is well known, see, e.g., \cite{Jonquiere1889,Panzer2017}.

\begin{re}
We point out that the branch cut for log was different in \cite{Panzer2017} and a careful re-evaluation is needed here. Also, his ordering of the summation indices in the definition of multiple polylogarithm \eqref{equ:Polylog} is opposite to ours. For convenience of the reader, we will give a self-contained proof of Lemma~\ref{lem:PanzerDepth1Red}. The apparent main difference is that $\log(-z)$ in \cite{Panzer2017} should be replaced by $\log z-\pi i$ in our paper.
\end{re}

\begin{lem} \label{lem:PanzerDepth1Red}
For any $k\in\N$ and $z\in \CC\setminus[0,\infty)$ we have
\begin{align}\label{equ:PolylogDepth1Red1}
\Li_k(z)+(-1)^k\Li_k(1/z)
=&\, \pi i\frac{\log^{k-1} z}{(k-1)!} +2\sum_{j=0}^{\lfloor k/2 \rfloor} \zeta(2j)\frac{\log^{k-2j} z}{(k-2j)!} \\
=&\, -\frac1{k!}\sum_{j=0}^k \binom{k}{j}  B_j(2\pi i)^j \log^{k-j} z  \label{equ:PolylogDepth1Red2}
\end{align}
where $\zeta(0)=-1/2$ and $B_j$ are Bernoulli numbers. In particular, for all $x>0$,
\begin{equation}\label{equ:RealPolylogDepth1Red}
\Li_k(x\pm 0i)+(-1)^k\Li_k(1/(x\pm 0i))
= \pm \pi i\frac{\log^{k-1} x}{(k-1)!} +2\sum_{j=0}^{\lfloor k/2 \rfloor} \zeta(2j)\frac{\log^{k-2j} x}{(k-2j)!}.
\end{equation}
\end{lem}

\begin{proof}
When $k=1$,
\begin{equation*}
\Li_1(z)-\Li_1(1/z)=-\log(1-z)+\log(1-1/z)=-\log z+\pi i.
\end{equation*}
When $k\ge 2$, by using the relations $d/dz \big(\Li_k(z)\big)=\Li_{k-1}(z)/z$ and $\Li_k(1)=\zeta(k)$ we can prove \eqref{equ:PolylogDepth1Red1} by induction immediately. Then \eqref{equ:PolylogDepth1Red2} follows easily from the Euler identity
\begin{equation*}
2\zeta(2j)=-(2\pi i)^{2j} B_{2j}/(2j)!.
\end{equation*}

For \eqref{equ:RealPolylogDepth1Red}, we only need to consider $\Li_k(x-0i)=\Li_k\big(e^{2\pi i} (x+0i)\big)$.
We first rewrite \eqref{equ:PolylogDepth1Red2} using umbral calculus:
\begin{align}\label{equ:calB}
\Li_k(x+0i)+(-1)^k\Li_k(1/(x+0i))=-\calB_k(x):=-\frac1{k!}(\log x+2\pi i B)^k
\end{align}
where $B^j$ should be understood as $B_j$. Thus,
\begin{align*}
\Li_k(x-0i)+(-1)^k\Li_k(1/(x-0i))
=-\frac1{k!}(\log x+2\pi i (B+1))^k=-\frac1{k!}(\log x+2\pi i B(1))^k
\end{align*}
where $B(1)^j=B_j(1)$ is the value of the $j$-th Bernoulli polynomial at 1. Here, we have used the fact that
$B_j(t)=(B+t)^j$. Noticing that $B_j(1)=B_j$ for all $j\ne 1$ and $B_1(1)=-B_1=1/2$, we quickly obtain \eqref{equ:RealPolylogDepth1Red}.
\end{proof}

\begin{cor}\label{cor:x>1}
For all real variable $x>1$ and positive integer $k\ge 2$, we have $\Li_k(x\pm 0i)=R_k(x)\pm I_k(x) \pi i$ where
\begin{equation*}
R_k(x)=-(-1)^k\Li_k(1/x)+2\sum_{j=0}^{\lfloor k/2 \rfloor} \zeta(2j)\frac{\log^{k-2j} x}{(k-2j)!},
\quad\text{and}\quad
I_k(x)=\frac{\log^{k-1} x}{(k-1)!}
\end{equation*}
are both real-valued functions of $x$.
\end{cor}

\begin{thm}\label{thm:TrivarDepth1}
For any fixed positive integers $k\ge 1$ and $j\ge 1$, the coefficients of the Maclaurin series of
\begin{equation*}
\frac{-x}{\Li_k(x)} \quad\text{and}\quad  \frac{x}{(1-x)^j\Li_k(x)}
\end{equation*}
are eventually positive and all positive, respectively. Furthermore, the coefficients
\begin{equation*}
C_{n}^{0;k;1}<0\ \forall n\ge 1
\quad\text{and}\quad
C_{n}^{j;k;1}>\binom{n+j-1}{n}\frac{1}{\zeta(k)}\ \forall j\ge 1, k\ge 2, n\ge 0.
\end{equation*}
\end{thm}

\begin{proof}
When $k=1$ the theorem is reduced to Theorem~\ref{thm:NorlundType} with $\ell=1$ so that we will assume $k\ge 2$ for the rest of this proof.

For any $k\ge 2$ and $n\ge 1$, we have
\begin{align*}
-C_{n}^{0;k;1}= &\, \Res_{z=-1} \frac{-1}{(z+1)^n\Li_k(z+1)}   \\
=&\, \frac{1}{2\pi i}\left( \int_\eps^R +\int_{C_R}+\int_R^\eps+ \int_{C_\eps} \right) \frac{-dz}{(z+1)^n\Li_k(z+1)}
\end{align*}
where the contour here is the same as in Figure~\ref{fig:C}. Note that $|\Li_k(z+1)|\to \infty$ as $|z|\to \infty$ by Lemma~\ref{lem:PanzerDepth1Red}
and $\Li_k(z+1)\to \Li_k(1)=\zeta(k)$ as $z\to 0$ by definition. Thus
\begin{align*}
-C_{n}^{0;k;1}=&\, \frac{1}{2\pi i} \int_0^\infty \left(\frac{1}{\Li_k(x+1-0i)}-\frac{1}{\Li_k(x+1+0i)}  \right)\frac{dx}{(x+1)^n}\\
 =&\, \frac{1}{2\pi i} \int_1^\infty \left(\frac{1}{\Li_k(x-0i)}-\frac{1}{\Li_k(x+0i)}  \right)\frac{dx}{x^n}\\
=&\, \frac{1}{2\pi i} \int_1^\infty \left(\frac{1}{R_k(x)-I_k(x)\pi i}-\frac{1}{R_k(x)+I_k(x)\pi i}  \right)\frac{dx}{x^n}
\end{align*}
by Corollary~\ref{cor:x>1}, where $(k-1)! I_k(x)=\log^{k-1} x$.
Thus for all $k>1$ we have
\begin{align*}
-C_{n}^{0;k;1}=&\, \int_1^\infty\frac{I_k(x)\, dx}{(R_k(x)^2+\pi^2 I_k(x)^2)x^n}>0.
\end{align*}
Similarly, for all $n\ge 1$,
\begin{align*}
C_{n}^{1;k;1}=&\, \int_1^\infty\frac{I_k(x)\, dx}{(R_k(x)^2+\pi^2 I_k(x)^2)(x-1) x^n}+\frac{1}{2\pi i}\lim_{\eps\to 0^+}\int_{C_\eps} \frac{-dx}{x(x+1)^n\Li_k(x+1)}\\
=&\, \int_1^\infty\frac{I_k(x)\, dx}{(R_k(x)^2+\pi^2 I_k(x)^2)(x-1) x^n}+\frac{1}{\zeta(k)}
>\frac{1}{\zeta(k)}.
\end{align*}
Since $C_{0}^{1;k;1}=1>1/\zeta(k)$,
taking any $j\ge 1$, by the proof of Lemma~\ref{lem:eventualPos}, we see that
\begin{equation*}
C_{n}^{j;k;1}>\frac{1}{\zeta(k)}\sum_{m=0}^n \binom{m+j-2}{j-2} =\frac{\binom{n+j-1}{j-1}}{\zeta(k)}
\end{equation*}
for all $n\ge 0$.

This completes the proof of Theorem~\ref{thm:TrivarDepth1}.
\end{proof}

\begin{cor}\label{cor:ratZeta(k)}
For any $k,n\in\N$ with $k\ge 2$, there is some $\nu_{n,k}\in\Q$ such that
\begin{equation*}
    \frac{1}{\zeta(k)}=\nu_{n,k}-\int_1^\infty\frac{I_k(x)\, dx}{(R_k(x)^2+\pi^2 I_k(x)^2)(x-1) x^n},
\end{equation*}
with the integral approaching $0$ as $n\to\infty$.
\end{cor}
\begin{proof}
This follows immediately from the fact the $\nu_{n,k}=C_{n}^{1;k;1}$ is a rational number.
\end{proof}

\subsection{Depth two case: double polylogarithms}

When the depth of $\bfk$ is 2, Lemma~\ref{lem:PanzerRed} has the following explicit form.

\begin{lem} \label{lem:PanzerDepth2Red} \emph{(cf. \cite[(3.2)]{Panzer2017})}
Let $a,b\in\N$ and $w=a+b$. If $(z,a)\ne (1,1)$ then
\begin{align*}
\Li_{a,b}(z)=-(-1)^w \Li_{a,b}(1/z)+&\, (-1)^b \sum_{\mu=b+1}^w (-1)^\mu \binom{\mu-1}{b-1} \zeta(\mu) \calB_{w-\mu}(z)\\
-\Li_{a+b}(z)+&\,(-1)^a \sum_{\mu=a}^w \binom{\mu-1}{a-1} \Li_\mu(1/z) \calB_{w-\mu}(z)
\end{align*}
where $\calB_j(z)$ is defined by \eqref{equ:calB}.
\end{lem}

\begin{thm}\label{thm:TrivarDepth2}
For any fixed positive integers $a,b,j\ge 1$, the coefficients of the Maclaurin series of
\begin{equation*}
\frac{-x^2}{\Li_{a,b}(x)} \quad\text{and}\quad  \frac{x^2}{(1-x)^j\Li_{a,b}(x)}
\end{equation*}
are both eventually positive. Furthermore, the coefficients
\begin{equation*}
C_{n}^{j;a,b;1}>\binom{n+j-1}{n}\frac{1}{\zeta(a,b)}\quad \forall a\ge 2, b,j\ge 1,  n\gg 0.
\end{equation*}
\end{thm}

\begin{proof}
If $a=b=1$ then the theorem is reduce to Corollary~\ref{cor:multilogGregory} and Corollary~\ref{cor:multilogNorlund} with $k=2$.
Thus we may assume $w\ge 3$ for the rest of the proof.

By the same argument as in Section~\ref{sec:contour}, we see that for $n\ge 1$
\begin{align*}
-C_{n}^{0;a,b;1}= &\, \Res_{z=-1} \frac{-1}{(z+1)^{n-1}\Li_{a,b}(z+1)}   \\
=&\, \frac{1}{2\pi i}\left( \int_\eps^R +\int_{C_R}+\int_R^\eps+ \int_{C_\eps} \right) \frac{-dz}{(z+1)^{n-1}\Li_{a,b}(z+1)}
\end{align*}
where the contour here is the same as in Figure~\ref{fig:C}. Note that $|\Li_{a,b}(z+1)|\to \infty$ as $|z|\to \infty$ by Lemma~\ref{lem:PanzerDepth2Red}
and, if $a\ge 2$, $\Li_{a,b}(z+1)\to \Li_{a,b}(1)=\zeta(a,b)$ as $z\to 0$ by definition. Further,
$|\Li_{1,b}(z+1)|\to \infty$ as $z\to 0$. Thus
\begin{align*}
-C_{n}^{0;a,b;1}=&\, \frac{1}{2\pi i} \int_0^\infty \left(\frac{1}{\Li_{a,b}(x+1-0i)}-\frac{1}{\Li_{a,b}(x+1+0i)}  \right)\frac{dx}{(x+1)^{n-1}}\\
 =&\, \frac{1}{2\pi i} \int_1^\infty \left(\frac{1}{\Li_{a,b}(x-0i)}-\frac{1}{\Li_{a,b}(x+0i)}  \right)\frac{dx}{x^n}\\
=&\, \frac{1}{2\pi i} \int_1^\infty \left(\frac{1}{R_{a,b}(x)-I_{a,b}(x)\pi i}-\frac{1}{R_{a,b}(x)+I_{a,b}(x)\pi i}  \right)\frac{dx}{x^{n-1}}\\
=&\, \int_1^\infty\frac{I_{a,b}(x)\, dx}{(R_{a,b}(x)^2+\pi^2 I_{a,b}(x)^2)x^{n-1}},
\end{align*}
where $\Li_{a,b}(x+0i)=R_{a,b}(x)+\pi i I_{a,b}$ for all $x>1$. By Lemma~\ref{lem:PanzerDepth2Red}, we see that
\begin{align*}
 I_{a,b}(x)=-\frac{\log^{w-1}x}{(w-1)!}-&\, (-1)^b \sum_{\mu=b+1}^{w-1} (-1)^\mu \binom{\mu-1}{b-1} \zeta(\mu) \frac{\log^{w-\mu-1}x}{(w-\mu-1)!}\\
-&\,(-1)^a \sum_{\mu=a}^{w-1} \binom{\mu-1}{a-1} \Li_\mu(1/z) \frac{\log^{w-\mu-1}x}{(w-\mu-1)!}.
\end{align*}
Hence,
\begin{align*}
 I_{a,b}(1)=\delta_{a\ge 2} (-1)^{w-b} \binom{w-2}{b-1} \zeta(w-1)-(-1)^a \binom{\mu-2}{a-1}\Li_{w-1}(1)=0
\end{align*}
since $w\ge 3$. For all $s<a$ we see that
\begin{align*}
g_s(x):= &\,\left(x \frac{d}{dx}\right)^j I_{a,b}(x)
=\frac{1}{\pi i} {\rm Im} \left[\left(x \frac{d}{dx}\right)^j \Li_{a,b}(x)\right]=I_{a-s,b}(x).
\end{align*}
Hence,
\begin{align*}
g_0(1)=g_1(1)=\cdots=g_{a-2}(1)=0,
\end{align*}
and $g_{a-1}(1):= \zeta(b)$ if $b\ge 2$. If $b=1$ then
\begin{align*}
    \lim_{x\to 1^+} g_{a-1}(x)=\lim_{x\to 1^+}  \Li_1(1/z) = \sum_{n=1}^\infty \frac1n = +\infty.
\end{align*}
This implies that there exists $r>1$ such that $I_{a,b}(x)>0$ over the interval $(1,2r)$.

Set $h(x):=I_{a,b}(x)/(R_{a,b}(x)^2+\pi^2 I_{a,b}(x)^2)$,
\begin{equation*}
\ga:=\int_1^r h(x)\, dx, \quad\text{and}\quad \beta:=\int_{2r}^\infty \frac{|h(x)|}{x^{2}}\, dx=\int_0^{1/(2r)} |h(1/x)|\, dx.
\end{equation*}
Noticing that $\lim_{x\to \infty} |h(x)|=0$ we know that both $\ga$ and $\beta$ are finite and positive.
We see that
\begin{align*}
-C_{n}^{0;a,b;1}> &\, \int_1^r \frac{h(x)}{x^{n-1}} \, dx-\int_{2r}^\infty \left|\frac{h(x)}{x^{n-1}}\right| \, dx\\
\ge &\, \frac{\ga}{r^{n-1}}-\int_0^{1/(2r)} |h(1/x)| x^{n-3}  \, dx\\
\ge &\, \frac{r\ga}{r^n} -  \frac{r^3\beta}{2^{n-3} r^n}=\frac{r\ga}{r^n} \left( 1-\frac{r^2\beta}{2^{n-3}\ga}\right)>0
\end{align*}
for all $n>3+\log_2(r^2\beta/\ga)$.

Completely similar ideas can be applied to bound $C_{n}^{1;a,b;1}$. The only difference is the possible additional contribution
by the contour integral around $0$ when $a\ge 2$, with the additional value $1/\zeta(a,b)$. The rest of the theorem now follows
from Lemma~\ref{lem:fjInduction}.

This completes the proof of Theorem~\ref{thm:TrivarDepth2}.
\end{proof}

\begin{cor}\label{cor:ratZeta(a,b)}
For any $a,b,n\in\N$ with $a\ge 2$, there is some $\nu_{n,a,b}\in\Q$ such that
\begin{equation*}
    \frac{1}{\zeta(a,b)}=\nu_{n,a,b}-\int_1^\infty\frac{I_{a,b}(x)\, dx}{(R_{a,b}(x)^2+\pi^2 I_{a,b}(x)^2)(x-1) x^n},
\end{equation*}
with the integral approaching $0$ as $n\to\infty$.
\end{cor}
\begin{proof}
This follows immediately from the fact the $\nu_{n,a,b}=C_{n}^{1;a,b;1}$ is a rational number.
\end{proof}

\begin{cor}\label{cor:iteratedInt}
Conjecture~\ref{conj:Trivar} holds for the case $j\ge 0$, $\dep(\bfk)=1$ and $\ell=2$.
\end{cor}
\begin{proof}
We only need the following fact which is an easy consequence of the shuffle relations of iterated integrals:
\begin{equation*}
 \Li_k(z)^2=\left[\int_0^z \left(\frac{dt}{t}\right)^{k-1} \frac{dt}{1-t}\right]^2
= \sum_{\substack{\mu,\nu\ge 1\\ \mu+\nu=2k}}
\bigg[\binom{\mu-1}{k-1}+\binom{\nu-1}{k-1}\bigg] \Li_{\mu,\nu}(z).
\end{equation*}
Then the behavior of the imaginary part of $\Li_k(x\pm 0i)^2$ are determined by that of the double polylarithms $\Li_{\mu,\nu}(x\pm 0i)$
and the proof of the corollary follows from the same argument as in the proof of Theorem~\ref{thm:TrivarDepth2}. Thus, we leave the details to the interested reader.
\end{proof}

\section{Concluding remarks}
In this paper, we first generalized the classical Cauchy numbers of the first and second kind, i.e., the Gregory coefficients and N\"olund numbers, to their higher order analogs. Two new phenomena appear for the higher order Gregory coefficients $G_n^{(\ell)}$. First, for each $\ell\ge 2$, some $G_n^{(\ell)}$ must vanish although the classical Gregory coefficients $G_n$'s are nonzero for all $n\ge1$. Second, the well-known alternating pattern of the $G_n$ continues to hold for $G_n^{(\ell)}$, but only for sufficiently large $n$, namely, for $n\ge \Gr(\ell)$, where, by computer-aided computation, $\Gr(\ell)$ seems to be roughly tripling as $\ell$ increases, the reason of which is still a mystery.

Replacing the logarithm by multiple polylogarithms, we defined a trivariate version of the Cauchy numbers of the first and second kind. From numerical experiments, we have found that the vanishing coefficient does not always exist but the alternating property, turning to eventual positiveness in the general case, should continue to hold (see Conjecture~\ref{conj:Trivar}). We then verified this conjecture in the polylogarithm and double polylogarithm cases. It is possible to extend these proofs to the depth three case using the explicit expression in \cite[(4.3)]{Panzer2017}. However, due to the implicit nature of the general result as presented in Lemma~\ref{lem:PanzerRed}, it is likely that new approach is need  to prove the general case of Conjecture~\ref{conj:Trivar}. However, as the proof of Corollary~\ref{cor:iteratedInt} shows, we only need to consider $\ell=1$ since the general case follows from this case by the shuffle relations of iterated integrals.

As a side remark, we notice that reciprocal of logarithm is intimately related to the distribution of primes and appears in the prime number theorem. We wonder if the reciprocal of multiple polylogarithms also plays some role in certain relevant studies of primes.

We would like to know if the general coefficients $C_{n}^{j;\bfk;\ell}$ defined by \eqref{equ:trivariateDefn} have any significant arithmetic properties such as congruences or if they are related to some other interesting objects in number theory such as those discovered in \cite{KanekoMatsusakaSeki2025}.

Finally, as far as we know, Corollaries~\ref{cor:ratZeta(k)} and ~\ref{cor:ratZeta(a,b)} provide the approximation of the reciprocals of zeta and double zeta values by rational numbers nontrivially for the first times. We wonder if these can be improved to obtain some nontrivial irrationality results concerning zeta and double zeta values.

\medskip\noindent
{\bf Declaration of Competing Interest.} The authors declare that there is no competing interest.

\medskip\noindent
{\bf Acknowledgments.} Both authors gratefully acknowledge the invitation by Professor Chengming Bai and the support of the Visiting Scholars Program at the Chern Institute of Mathematics and by Professors Li Lai, Shaoyun Yi and Huilin Zhu of Xiamen University and the support of the Tianyuan Mathematical Center in Southeast China (TMSE). This work commenced during these visits. The second author also wants to thank Dr. Marcus Jaiclin at The Bishop's School for his generous help with our numerical computation. Ce Xu is supported by the General Program of Natural Science Foundation of Anhui Province (Grant No. 2508085MA014).

\end{document}